\documentclass[a4paper, 12pt]{amsart}
\usepackage{amsmath,amssymb,amsthm}
\usepackage{amscd}
\usepackage{array,enumerate}

\newcommand{\ipa}[1]{\mathopen{\langle}#1\mathclose{\rangle}_A}

\newcommand{\ide}[2]{\mathopen{\langle}#1\mathclose{\rangle}_{#2}}
\newtheorem{Thm}{Theorem}
\newtheorem{Prop}[Thm]{Proposition}
\newtheorem{Lem}[Thm]{Lemma}
\newtheorem{Cl}[Thm]{Claim}

\theoremstyle{definition}
\newtheorem{Rem}[Thm]{Remark}

\newtheorem*{Ques}{Question}
\newcommand{\Cs}{C$^\ast$}

\newcommand{\rca}[1]{\mathop{\rtimes _{{\mathrm r}, #1}}}

\newcommand{\IB}{\mathbb B}

\newcommand{\IK}{\mathbb K}
\newcommand{\IM}{\mathbb M}
\newcommand{\IN}{\mathbb N}

\newcommand{\IT}{\mathbb T}
\newcommand{\IZ}{\mathbb Z}

\newcommand{\fH}{\mathfrak H}

\newcommand{\fs}{\mathfrak s}

\newcommand{\ft}{\mathfrak t}

\newcommand{\fl}{\mathfrak l}
\newcommand{\fk}{\mathfrak k}

\newcommand{\cZ}{\mathcal Z}

\newcommand{\cI}{\mathcal I}
\newcommand{\cM}{\mathcal M}

\newcommand{\acts}{\curvearrowright}

\newcommand{\ad}{\mathrm{Ad}}

\DeclareMathOperator{\bigfp}{\lower0.25ex\hbox{\LARGE $\ast$}}

\title[Amenable actions on ill-behaved simple \Cs-algebras]
{Amenable actions on ill-behaved simple \Cs-algebras}
\author{Yuhei Suzuki}
\address{Department of Mathematics, Faculty of Science, Hokkaido University,
Kita 10, Nishi 8, Kita-Ku, Sapporo, Hokkaido, 060-0810, Japan}

\email{yuhei@math.sci.hokudai.ac.jp}
\begin{document}
\maketitle

\begin{abstract}
By combining R{\o}rdam's construction and the author's previous construction,
we provide the first examples of amenable actions of non-amenable groups on simple separable nuclear \Cs-algebras
that are neither stably finite nor purely infinite.
For free groups, we also provide unital examples.
We arrange the actions so that the crossed products are still simple
with both a finite and an infinite projection.

\end{abstract}
\section{Introduction}
The symmetrical structure of operator algebras is significant for deep understandings
of operator algebras and related phenomena.
Quite successful milestones include Tomita--Takesaki theory, Connes's classification theory of injective von Neumann algebras,
and Popa's deformation/rigidity theory.

The aim of this paper is to shed light on the symmetrical structure of simple \Cs-algebras not covered
in the expected type classification analogue of the Murray--von Neumann type classification of factors,
that is, simple \Cs-algebras containing both a finite and an infinite projection. Simple (separable nuclear) \Cs-algebras of this exceptional type have been constructed by R{\o}rdam \cite{Ror}.
We point out that such mysterious simple \Cs-algebras do not appear in the current range of classification theory of unital simple nuclear \Cs-algebras (see e.g., \cite{CGS}, \cite{Win}).
(In fact such a simple separable nuclear \Cs-algebra $A$ has the same Elliott invariant and the Kasparov class as the Kirchberg algebra $A\otimes \mathcal{O}_\infty$.)
 
We continue further investigation of the following fundamental question studied in \cite{Suzeq}, \cite{OS}, \cite{Suzf}, \cite{SuzP}.
\begin{Ques}
Which non-amenable group admits an amenable action on which simple \Cs-algebra?
\end{Ques}
We construct amenable actions on simple \Cs-algebras of the forementioned exceptional type,
by combining R{\o}rdam's construction \cite{Ror} together with our construction of amenable actions in \cite{Suzf}.

\subsection*{Motivations}
Contrary to expectations by experts (see e.g., \cite{AnaT}, \cite{BO}),
unlike the von Neumann algebras \cite{Ana79},
it is shown in \cite{Suzeq} that non-amenable groups can act amenably on highly non-commutative \Cs-algebras,
remarkably on simple \Cs-algebras.
This discovery leads to unexpected developments in classification theory.
Amenability of \Cs-dynamical systems was first used in \cite{Suz21} for classification type results.
Now, by the major breakthrough of Gabe--Szab\'o 
(the classification theorem of amenable actions on Kirchberg algebras \cite{GS2}),
there is no doubt that amenability of \Cs-dynamical systems, not of the acting group, is a primal ingredient for the classification of \Cs-dynamical systems.
Therefore it is notable that whatever the group is,
it can act amenably on ill-behaved simple (nuclear) \Cs-algebras.
In other words, some of these \Cs-algebras could still have a rich symmetrical structure.
These days there are some attempts of (partial) classifications and finer understandings of ill-behaved (i.e., failures of good comparison properties) simple separable nuclear \Cs-algebras; see e.g., \cite{ELN}, \cite{HP1}, \cite{HP2}, \cite{LR}, \cite{Niu}, \cite{SuzIMRN}.
Our results will have an impact on the research in this direction.

\subsection*{Main results}
Now we present the main theorems of this paper.
\begin{Thm}\label{Thm:1}
Let $\Gamma$ be a countable group.
Then there is an amenable pointwise outer\footnote{An automorphism $\alpha$ on a \Cs-algebra $A$ is said to be outer,
if there is no unitary multiplier element $u\in \cM(A)$ with $\alpha=\ad(u)$. We say an action $\alpha \colon \Gamma \acts A$
is pointwise outer if $\alpha_g$ is outer for all $g\in \Gamma \setminus\{e\}$.} action $\alpha\colon \Gamma \acts A$ on a \Cs-algebra with the following properties.
\begin{enumerate}[\upshape(i)]
\item\label{item:a} $A$ is simple, separable, nuclear, and satisfies the universal coefficient theorem {\rm \cite{RS}}.
\item\label{item:b} $A$ has both a finite and an infinite nonzero projection.
\item\label{item:c} The reduced crossed product $A\rca{\alpha} \Gamma$ also satisfies the properties in {\rm (\ref{item:a})} and {\rm (\ref{item:b})}.
\end{enumerate}
Moreover, when $\Gamma$ is a free group, one can arrange $A$ to be unital.
\end{Thm}
The actions in Theorem \ref{Thm:1} have the following very different feature from
amenable actions on simple \Cs-algebras of the first kind \cite{Suzeq}, as well as the model actions in \cite{GS2}, that is, the infinite tensor product actions.
Recall that a simple \Cs-algebra $A$ is said to be \emph{tensorially prime} if $A\cong B \otimes C$ implies one of $B$ or $C$ is of type I.
Kirchberg's dichotomy (\cite{Rorbook}, Theorem 4.1.1) states that a simple exact tensorially non-prime \Cs-algebra
is either stably finite or purely infinite. 
By this dichotomy, the above theorem in particular gives
the first examples of amenable actions (of a non-amenable group) on tensorially prime simple \Cs-algebras.
This would be notable, as amenable \Cs-dynamical systems of exact groups
have rich dynamical systems on the central sequence algebras at least on the operator system level; cf.~characterizations of amenability of \Cs-dynamical systems in \cite{Suz21}, Theorem C and \cite{OS}, Section 4 (see also Remark \ref{Rem:central}).

It would be worth to compare the condition (\ref{item:c}) in Theorem \ref{Thm:1} (see also Theorem \ref{Thm:2} below)
with the following two results on the reduced crossed products of a discrete group $\Gamma$.

\begin{itemize}
\item
If $A$ is purely infinite simple and $\alpha\colon \Gamma \acts A$
is pointwise outer, then $A\rca{\alpha} \Gamma$ is again purely infinite simple.
This easily follows from the proof of Kishimoto's simplicity theorem \cite{Kis}, see e.g., \cite{SuzRig}, Lemma 6.3 for details.

\item 
If $A$ is unital, $\cZ$-stable, simple, nuclear; $\Gamma$ contains a non-commutative free group;
and the action $\alpha \colon \Gamma \acts A$ is pointwise outer and tracially amenable (a condition weaker than amenability),
then $A\rca{\alpha} \Gamma$ is purely infinite simple (\cite{GG}, Theorem C).
\end{itemize}
Compared to these facts, Theorem \ref{Thm:1} shows that when the underlying simple \Cs-algebra is ill-behaved,
the reduced crossed product can also be ill-behaved,
independent of the analytic property of the action and the combinatorial structure of the acting group.
This reveals a new insight on the classifiability of the reduced crossed products.

When $\Gamma$ is exact, a variant of the construction in Section 5 in \cite{Ror} also gives amenable actions on unital simple ill-behaved \Cs-algebras.
Note that by Corollary 3.6 of \cite{OS}, the exactness assumption cannot be removed.
\begin{Thm}\label{Thm:2}
Let $\Gamma$ be a countable exact group.
Then there is an amenable pointwise outer action $\alpha \colon \Gamma \acts A$ with the following properties.
\begin{enumerate}[\upshape(i)]
\item\label{item:d} $A$ is unital, simple, separable, and non-exact.
\item\label{item:e} $A$ is properly infinite, and contains a nonzero finite projection.
\item\label{item:f} $A\rca{\alpha} \Gamma$ also satisfies the properties in {\rm (\ref{item:d})} and {\rm (\ref{item:e})}.
\end{enumerate}
\end{Thm}

We point out that the theorems in particular give
the first examples of amenable actions (of non-amenable groups) on unital simple \Cs-algebras without pure infiniteness\footnote{Amenable actions on stably finite simple \Cs-algebras
are constructed in \cite{Suzf}, \cite{SuzP}. However currently all these examples are non-unital
by essential reasons.}.
They also provide a simple \Cs-algebra of the exceptional type (which can be arranged to be unital, separable, and nuclear)
with a non-trivial simple regular \Cs-subalgebra.
To the knowledge of the author, such examples were not known before.
Also, these particular crossed product presentations would be helpful to understand inclusions involving such \Cs-algebras; cf.~ \cite{SuzMAAN}.
\section{Preliminaries}
\subsection*{Notations}
We basically follow the notations of \cite{Ror}.
However we include $0$ in $\IN$
because it is more natural and convenient for our purpose.
This causes shifts of some indices in \cite{Ror}
(but certainly this does not require essential changes).

Let $\mathbb{S}^2$ denote the $2$-sphere.
Set $Z:=(\mathbb{S}^2)^{\IN}$.
For $n\in \IN$, let $\pi_n \colon Z \rightarrow \mathbb{S}^2$ be the $n$th coordinate projection.
Denote by $p\in \IM_2(C(\mathbb{S}^2))$ the Bott projection (see \cite{BlaK}, 9.2.10).
For $n\in \IN$, put $p_n:=\pi_n^\ast(p) \in \IM_2(C(Z))$.
For a non-empty finite subset $I\subset \IN$, let
$\Delta_I \colon Z \rightarrow Z^I$ denote the diagonal embedding.
Then, by identifying $\bigotimes_{i\in I} \IM_2$ with $\IM_{2^{|I|}}$ via a fixed isomorphism, we define
\[p_I:= \Delta_I^\ast(\bigotimes_{i\in I} p_i) \in \IM_{2^{|I|}} (C(Z)).\]

Put $\IK:=\IK(\ell^2(\IN\setminus \{0\}))$.
For a \Cs-algebra $A$ and $n \geq 1$, we identify the matrix amplification $\IM_n(A)$ with the corner $A\otimes \IK(\ell^2(\{1, \ldots, n\}))$ of $A \otimes \IK$ via the obvious isomorphism.
In particular we identify $A$ with the corner $A\otimes \IK(\ell^2(\{1\}))$.

For a stable \Cs-algebra $A$ (that is, $A \cong A\otimes \IK$) and a countable (non-empty) index set $\Lambda$,
we fix a family $(V_\lambda)_{\lambda\in \Lambda}$ of isometry elements in the multiplier algebra $\cM(A)$ with
 \[\sum_{\lambda \in \Lambda} V_\lambda V_\lambda^\ast =1\]
 in the strict topology.
Then, for
a bounded indexed family $(a_\lambda)_{\lambda\in \Lambda}$ in $\cM(A)$,
 we define
 \begin{equation} \label{eq:DSe}\bigoplus_{\lambda \in \Lambda} a_\lambda:=\sum_{\lambda\in \Lambda} V_\lambda a_\lambda V_\lambda^\ast \in \cM(A).
 \end{equation}
 This defines a unital $\ast$-homomorphism
 \begin{equation}\label{eq:DS}
 \bigoplus_{\lambda\in \Lambda} \colon \ell^\infty(\Lambda, \cM(A))\rightarrow \cM(A).
 \end{equation}
 Note that the unitary conjugacy class of the $\ast$-homomorphism in (\ref{eq:DS}) is independent of the choice of $(V_\lambda)_{\lambda \in \Lambda}$.
 Indeed if we employ another family $(W_\lambda)_{\lambda \in \Lambda}$,
 then the unitary element $U:=\sum_{\lambda\in \Lambda} W_\lambda V_\lambda^\ast \in \cM(A)$ witnesses the unitary conjugacy.
 In particular, for an indexed family $(p_\lambda)_{\lambda \in \Lambda}$ of projections in $\cM(A)$,
 the Murray--von Neumann equivalence class of $\bigoplus_{\lambda \in \Lambda} p_{\lambda} \in \cM(A)$ is independent of the choice of $(V_\lambda)_{\lambda \in \Lambda}$.
 We also note that after identifying $\ell^\infty(\Lambda, \cM(A))$ with $\cM(c_0(\Lambda) \otimes A)$ in the canonical way, the $\ast$-homomorphism in (\ref{eq:DS}) is strictly continuous.
\subsection*{Facts}
Villadsen's invention \cite{Vil1}, \cite{Vil2} is to combine Euler class obstructions together with the inductive limit construction
to produce ill-behaved simple (nuclear) \Cs-algebras.
Until now, basically this is the only known approach to produce counter-examples \cite{Ror}, \cite{Tom} of the original Elliott conjecture.

For two projections $p, q$ in a \Cs-algebra $A$, we write $p\sim q$ when they are Murray--von Neumann equivalent.
We write $p \precsim q$ when there is a projection $q_0\leq q$ with $p \sim q_0$.
A projection $p\in A$ is said to be properly infinite if $p\oplus p \precsim p$ in $A\otimes \IK$.
A projection $p\in A$ is said to be finite if it is not infinite, that is,
if there is no proper subprojection $p_0 \leq p$ with $p_0 \sim p$.

Based on Euler class obstructions, R{\o}rdam  \cite{Ror} observed the following criterion on the Murray--von Neumann comparison relation.
\begin{Prop}[\cite{Ror}, Proposition 4.5]\label{Prop:Ror}
Let $(I_\lambda)_{\lambda\in \Lambda}$ be a $($countable$)$ indexed family\footnote{We point out that the sets of finite subsets of $\IN$ appearing in \cite{Ror} e.g., $\Gamma_n$, $\Gamma$ should be replaced by \emph{indexed} families (or \emph{multisets}) of finite subsets, due to the possibility of overlaps. Cf.~the proof of Theorem \ref{Thm:1} below.} 
of finite subsets of $\IN$ with the following property:
There is an injective map $N \colon \Lambda \rightarrow \IN$ with $N(\lambda) \in I_\lambda$ for $\lambda \in \Lambda$.
Then we obtain the relation
\[1_{C(Z)}\not\precsim \bigoplus_{\lambda \in \Lambda} p_{I_\lambda} \quad {\rm in}\quad \cM(C(Z)\otimes \IK).\]
\end{Prop}

Also, for any $n\in \IN$, one has
\begin{equation}\label{eq:trivial}
1_{C(Z)}\precsim p_n \oplus p_n
\end{equation}
in $C(Z) \otimes \IK$
because $1_{C(\mathbb{S}^2)} \precsim p \oplus p$ in $C(\mathbb{S}^2)\otimes \IK$.
 These two special properties of the projections $p_I$ are critical in R{\o}rdam's constructions \cite{Ror}.
 
 The next lemma is a version of the key observation, Lemma 6.4 in \cite{Ror}.
 In \cite{Ror}, it is stated for $\IZ$-actions, but the same proof gives the following version.
 For completeness, we give a proof.
 \begin{Lem}\label{Lem:fp}
 Let $C$ be a \Cs-algebra. Let $\gamma \colon \Gamma \acts C$ be an action.
 Assume that two projections $p, q\in C$ satisfy the relations
\begin{equation}\label{equation:1}p\precsim q\oplus q,
\end{equation}
 \begin{equation}\label{equation:2}
 p\not\precsim \bigoplus_{g\in \Gamma} \gamma_g(q)\end{equation}
 in $\cM(C\otimes \IK)$.
 Then $q$ is not properly infinite in $C\rca{\gamma} \Gamma$.
 \end{Lem}
We note that
an infinite projection in a simple \Cs-algebra is automatically properly infinite \cite{Cun} (see also \cite{BlaK}, Proposition 6.11.3).
 Hence, when $C\rca{\gamma}\Gamma$ is simple, Lemma \ref{Lem:fp} in particular gives a sufficient condition for the finiteness of $q$.
\begin{Rem}\label{Rem:finite}
The condition in (\ref{equation:2}) is equivalent to the following condition:
 For any finite subset $F\subset \Gamma$, one has
\[p\not\precsim \bigoplus_{g\in F} \gamma_g(q)\]
in $C\otimes \IK$. See Lemma 4.4 in \cite{Ror}.
\end{Rem}

\begin{proof}[Proof of Lemma \ref{Lem:fp}]

Let $(e_{g, h})_{g, h\in \Gamma}$ denote the matrix unit of $\IK(\ell^2(\Gamma))$ given by $e_{g, h}(\delta_k)=\delta_{h, k} \cdot \delta_g$ for $g, h, k\in \Gamma$.
Then one has an embedding
\[\Phi \colon C\rca{\gamma} \Gamma \rightarrow \cM(C\otimes \IK(\ell^2(\Gamma)))\]
given by
\[\Phi(cs):=\sum_{g\in \Gamma} \gamma^{-1}_g(c)\otimes e_{g, s^{-1}g}\]
for $c\in C$ and $s\in \Gamma$.
(To see that $\Phi$ is indeed a well-defined embedding of the reduced crossed product,
observe that any faithful $\ast$-representation $\pi \colon C \rightarrow \IB(\fH)$ on a Hilbert space $\fH$
gives rise to the embedding $\tilde{\pi} \colon \cM(C\otimes \IK(\ell^2(\Gamma))) \rightarrow \IB(\fH\otimes \ell^2(\Gamma))$
given by $\tilde{\pi}(c \otimes x)=\pi(c)\otimes x$ for $c\in C$, $x\in \IK(\ell^2(\Gamma))$.
Then, the composite $\tilde{\pi} \circ \Phi$ is equal to the left regular covariant representation of $\gamma \colon \Gamma \acts C$ associated to $\pi$.
Thus $\tilde{\pi} \circ \Phi$, hence $\Phi$, is an embedding.)

Let $r \in \IK(\ell^2(\Gamma))$ be a minimal projection.
Then, because
\[\Phi(q)\sim \bigoplus_{g\in \Gamma} \gamma_g(q),\]
 the relations (\ref{equation:1}) and (\ref{equation:2}) in the assumption respectively imply
\[p\otimes r \precsim \Phi(q)\oplus \Phi(q), \quad p\otimes r \not\precsim \Phi(q)\]
 in $\cM(C\otimes \IK(\ell^2(\Gamma)))$.
These two relations force 
\[\Phi(q)\oplus \Phi(q) \not\precsim \Phi(q) \quad {\rm in}~ \cM(C\otimes \IK(\ell^2(\Gamma))),\]
whence $q\oplus q \not \precsim q$ in $(C\rca{\gamma}\Gamma)\otimes \IK$.
\end{proof}
For a subset $\mathcal{S}\subset A$ of a \Cs-algebra $A$,
we write $\ide{\mathcal{S}}{A}$ for the (closed, self-adjoint) ideal of $A$ generated by $\mathcal{S}$.
Recall that an element $a \in A$ is said to be full in $A$ if it satisfies $\ide{\{a\}}{A}=A$.
More generally, for a subset $\mathcal{S} \subset \cM(A)$, we set
$\ide{\mathcal{S}}{A}:=\ide{\mathcal{S}}{\cM(A)} \cap A$ $(=\ide{\mathcal{S}\cdot A}{A})$.

 For basic facts on amenable actions on non-commutative \Cs-algebras,
 we refer the reader to \cite{OS}.
 
 \section{Ingredients of the constructions}\label{section:ing}
In this section, we introduce the main ingredients of the construction of the desired actions in Theorem \ref{Thm:1}.
Most of them will be also used in the proof of Theorem \ref{Thm:2}.
\subsection*{Ingredients of our construction}
\begin{description}

\item[A sequence $\Lambda_n\subset \IN$]
We pick pairwise disjoint subsets $(\Lambda_n)_{n\in \IN}$ of $\IN$ with $0\in \Lambda_0$ and $|\Lambda_n|=|\IN|$ for $n \geq 1$.

\item[A generating subset $S\subset \Gamma$]
Choose a subset $S\subset \Gamma$ which generates $\Gamma$ as a semigroup with $e\in S$.
(Unlike the construction in \cite{Suzf}, we do not require the finiteness of $S$, which was necessary to produce a tracial weight.
In particular one can choose $S=\Gamma$.
This particular choice will be employed in the proof of Theorem \ref{Thm:2}.)

\item[A map $\nu \colon S \times \IZ \times \IN \rightarrow \IN$]
We pick an injective map
\[\nu\colon S \times \IZ \times \IN \rightarrow \IN\]
satisfying $\nu(S \times \IZ \times \Lambda_n)\subset \Lambda_{n+1}$ for $n\in \IN$.
\item[Subsets $I(s, n) \subset \IN$ indexed by $S\times \IN$]
For $(s, n)\in S\times \IN$, we set
\[I(s, n):=\nu(s, n, \{0, \ldots, n\}) \subset \IN.\]

\item[A dense sequence $((c_{j, n})_{n\in \IN})_{j\in \IN}$ in $Z$]
We fix a dense sequence $((c_{j, n})_{n\in \IN})_{j\in \IN}$ in $Z=(\mathbb{S}^2)^\IN$ (so each $c_{j, n}$ is in $\mathbb{S}^2$).

\item[Continuous maps $\rho_{s, j}\colon Z \rightarrow Z$ indexed by $S\times \IZ$]
For $(s, n)\in S\times \IZ$ and $z=(x_n)_n \in Z=(\mathbb{S}^2)^\IN$, we define $\rho_{s, j}(z)\in Z$ as follows. For $n\in \IN$, we define its $n$th entry to be
\begin{align}\label{eq:rho}
\rho_{s, j}(z)_n:=\left\{
\begin{array}{ll}
x_{\nu(s, j, n)} & \text{if } n > j,\\
c_{j, n} & \text{if } n \leq  j.
\end{array}
\right.
\end{align}
Clearly each $\rho_{s, j}$ is continuous.
Note that the second case in (\ref{eq:rho}) does not occur when $j<0$.

\item[$\ast$-endomorphisms $\psi_{s, j}$ on $C(Z)\otimes \IK$ indexed by $S\times \IZ$]
We fix a $\ast$-isomorphism $\tau \colon \IK\otimes \IK \rightarrow \IK$.
Then, for each $(s, j)\in S\times \IZ$, we define a $\ast$-homomorphism
\[\psi_{s, j} \colon C(Z) \otimes \IK \rightarrow C(Z)\otimes \IK\]
as follows.
For $f\in C(Z)\otimes \IK$ and $z \in Z$, set
\begin{align}\label{eq:psisj}
\psi_{s, j}&(f)(z):=\left\{
\begin{array}{ll}
f(\rho_{s, j}(z))& \text{if } j <0,\\
\tau(f(\rho_{s, j}(z))\otimes p_{I(s, j)}(z)) & \text{if } j \geq 0.
\end{array}
\right.
\end{align}
\item[A subset $\kappa_{s, j}(I) \subset \IN$ for $s\in S, j \in \IZ, I \subset \IN$]
For any non-empty finite subset $I\subset \IN$ and any $(s, j)\in S\times \IZ$, we define
\[\kappa_{s, j}(I):= \left\{
\begin{array}{ll}
\nu(s, j, I)& \text{if } j < 0,\\
\nu(s, j, I \cup \{0, 1, \ldots, j\}) &\text{if } j \geq 0.
\end{array}
\right.\]
\end{description}
~\\

\begin{Rem}[Some properties of $\psi_{s, j}$]~\\

\begin{itemize}
\item
For any non-empty finite subset $I\subset \IN$ and $(s, j)\in S\times \IZ$,
 the relation
\begin{equation}\label{eq:MvN}
\psi_{s, j}(p_I)\sim p_{\kappa_{s, j}(I)}
\end{equation}
holds true in $C(Z) \otimes \IK$ by the definitions (\ref{eq:rho}) and (\ref{eq:psisj}) (see Lemma 5.4 of \cite{Ror} for details).
\item
By the density of $((c_{j, n})_{n\in \IN})_{j\in \IN}$ in $Z$,
for any $f\in C(Z)\otimes \IK$ and any $s\in S$,
one has infinitely many (non-negative) $j\in \IZ$ with
\begin{equation}\label{eq:full}
\min_{z\in Z}\|\psi_{s, j}(f)(z) \| \geq \frac{\|f\|}{2}.
\end{equation}
The inequality (\ref{eq:full}) is crucial to obtain simple \Cs-algebras in the construction below.

\item Each $\psi_{s, j}$ admits a strictly continuous extension
\[\psi_{s, j}\colon \cM(C(Z) \otimes \IK) \rightarrow \cM(C(Z) \otimes \IK).\]
Indeed, after identifying $\cM(C(Z)\otimes \IK)$ with the \Cs-algebra of all bounded, strictly continuous functions $f\colon Z\rightarrow \cM(\IK)$ in the canonical way,
the same formula as in (\ref{eq:psisj}) defines the desired extension.
\end{itemize}
\end{Rem}

\begin{description}
\item[The \Cs-algebras $C:=C_0(\Gamma\times Z)\otimes \IK$ and $\cM(C)$]
These two \Cs-algebras are the basic building blocks of our construction.
Throughout the article, we identify $\cM(C)$
with $\ell^\infty(\Gamma, \cM(C(Z)\otimes \IK))$ in the obvious way.

\item[An amenable action $\beta \colon \Gamma \acts C$]

We equip the \Cs-algebra $C$ with the right translation action $\beta\colon \Gamma \acts C$ as follows.
For $g\in \Gamma$ and $(f_h)_{h\in \Gamma} \in \cM(C)$, set
\begin{equation}\label{eq:beta}
\beta_g((f_h)_{h\in \Gamma}):=(f_{hg})_{h\in \Gamma}.
\end{equation}
For the multiplier algebra of a $\Gamma$-\Cs-algebra,
we equip it with the induced $\Gamma$-action.
Clearly the $\ast$-homomorphism $c_0(\Gamma) \rightarrow \cM(C)$ given by $\delta_g \mapsto \delta_g \otimes 1$; $g\in \Gamma$,
is non-degenerate and $\Gamma$-equivariant with respect to the right translation action, and its image is contained in the center.
Thus $\beta$ is (strongly) amenable.

\item[A $\ast$-homomorphism $\psi\colon C \rightarrow \cM(C)$] 
Define a $\ast$-homomorphism
\[\psi\colon C \rightarrow \cM(C)\]
as follows. For $f \in C$ and $g\in \Gamma$, set
\begin{equation}\label{eq:psi}
\psi(f)_g:= \bigoplus_{(s, j)\in S \times \IZ}
 \psi_{s, j}(f_{s^{-1}g}).
 \end{equation}
\end{description}
Here it is important to use the \emph{same} indexed family $(V_{s, j})_{(s, j)\in S\times \IZ}$ of isometry elements, independent of $g\in \Gamma$,
to define the right hand side of (\ref{eq:psi}) (see (\ref{eq:DSe})).
Indeed, thanks to this, $\psi$ is $\Gamma$-equivariant.

Observe that $\psi$ has a strictly continuous extension
\[\psi \colon \cM(C) \rightarrow \cM(C),\]
because each $\psi_{s, j}$ admits a strictly continuous extension, and hence the same formula as in (\ref{eq:psi}) (but with the extended $\psi_{s, j}$'s) defines the desired extension.

\section{Theorem \ref{Thm:1}}
With the ingredients in Section \ref{section:ing}, we now prove Theorem \ref{Thm:1}.

For a \Cs-algebra $A$, denote by $1_A$ the unit of $\cM(A)$.
\begin{proof}[Proof of Theorem \ref{Thm:1}]
We first show the following claim.
\begin{Cl}\label{Cl:1}
One has
$\psi(1_C) \sim 1_C$ in the fixed point algebra $\cM(C)^\Gamma$.
\end{Cl}
To see this, note first that $\cM(C)^\Gamma \cong \cM(C(Z)\otimes \IK)$.
Indeed the evaluation map at any $g\in \Gamma$ gives such an isomorphism.
Also, as the strictly continuous extension of $\psi_{s, j}$ is unital for all $s\in S$ and $j<0$, one has
$1_C \oplus 1_C \precsim \psi(1_C)$
in $\cM(C)^\Gamma$.
Hence $\psi(1_C)$ is properly infinite and full in $\cM(C)^\Gamma \cong \cM(C(Z) \otimes \IK)$.
Thus Claim \ref{Cl:1} follows from Lemma 4.3 of \cite{Ror}.

By Claim \ref{Cl:1}, one can choose $W\in \cM(C)^\Gamma$ with $W W^\ast=1_C$, $W^\ast W=\psi(1_C)$.
Then
\begin{equation}\label{eq:phi}
\varphi:=\ad(W)\circ \psi
\end{equation}
defines a unital, $\Gamma$-equivariant, and strictly continuous $\ast$-endomorphism on
$\cM(C)$.
Note that for any $s\in S$, $g\in \Gamma$, and $f\in \cM(C)$ with $f_g \neq 0$,
infinitely many summands of $\psi(f)_{sg}$ in the definition (\ref{eq:psi}) is of norm at least $\|f\|/2$ by the inequality (\ref{eq:full}),
and hence $\varphi(f)_{sg}$ is not in $C(Z)\otimes \IK$.
Thus $\varphi$ is injective and satisfies
\begin{equation}\label{equation:int}
\varphi(\cM(C)) \cap C =\{0\}.
\end{equation}
We also note that for $f\in C$ and $g\in \Gamma$, when $f_g$ is full in $C(Z)\otimes \IK$
then $\varphi(f)_{sg}$ is full in $\cM(C(Z)\otimes \IK)$ for all $s\in S$.
Indeed, by the fullness of $f_g$, one has $\delta_g \otimes 1_{C(Z)} \in \ide{f}{C_0(\Gamma \times Z)}$.
Then by (\ref{eq:psisj}), (\ref{eq:psi}) and (\ref{eq:phi}),
one has $1_{C(Z)\otimes \IK} \precsim \varphi(\delta_g \otimes 1_{C(Z)})_{sg}$, which proves the fullness of  $\varphi(f)_{sg}$ in $\cM(C(Z)\otimes \IK)$.

Put
\begin{equation}\label{eq:B}
B:=\varinjlim_{n\in \IZ}(\cM(C), \varphi).
\end{equation}
For $n\in \IZ$, let $\mu_{n}\colon \cM(C) \rightarrow B$ denote the $n$th canonical map.
Let $D$ be the \Cs-subalgebra of $B$ generated by the subsets $\mu_{j}(C)$; $j\in \IZ$.
Note that for $j, k\in \IZ$ with $j \leq k$, one has
\begin{equation}\label{eq:muj}
\mu_j(C)\cdot \mu_k(C)=\mu_k(\varphi^{k-j}(C))\cdot \mu_k(C) \subset \mu_k(C).
\end{equation}

Clearly $D$ is separable. Moreover the following claim holds true.
\begin{Cl}\label{Cl:2}
The \Cs-algebra $D$ is nuclear and satisfies the universal coefficient theorem \cite{RS}.
\end{Cl}
For any two integers $m, n\in \IZ$ with $m\leq n$,
set 
\[D_{m, n}:= \mu_m(C)+ \mu_{m+1}(C) + \cdots +\mu_n(C).\]
By the relations (\ref{eq:muj}), each $D_{m, n}$ is a \Cs-subalgebra of $D$.
We next show that $D_{m, n}$ is of type I.
This follows by the same argument as in \cite{Ror}, Lemma 6.6.
Here we give a proof, as the following key short exact sequence will be also used in the proof of the amenability of actions.

\begin{Cl}\label{Cl:D}
For $m < n$, one has the $($split$)$ short exact sequence
\begin{equation}\label{equation:extD}
\{0\}\rightarrow D_{m+1, n} \rightarrow D_{m, n} \rightarrow C \rightarrow \{0\},\end{equation}
where the first map is the inclusion map,
and the second map comes from the isomorphism
$\Phi_{m, n} \colon C\rightarrow D_{m, n}/ D_{m+1, n}$
given by composing $\mu_{m}$ with the quotient map $D_{m, n} \rightarrow D_{m, n}/D_{m+1, n}$.
\end{Cl}
To prove Claim \ref{Cl:D}, it suffices to show that $\Phi_{m, n}$ is really an isomorphism.
The surjectivity of $\Phi_{m, n}$ is clear.
For the injectivity of $\Phi_{m, n}$, it suffices to show the next relation.

\begin{equation}\label{eq:D}
D_{m+1, n} \cap \mu_{m}(C) =\{0\}.
\end{equation}

Note that for $k\in \IZ$ and $l\in \IN$, one has
$\mu_k=\mu_{k+l} \circ \varphi^l$.
Hence
\[D_{m+1, n} \cap \mu_{m}(C)=\mu_n((C+\varphi(C)+ \cdots +\varphi^{n-m-1}(C)) \cap \varphi^{n-m}(C)).\]
Thus, to show (\ref{eq:D}), it suffices to show
\[(C+\varphi(C)+ \cdots +\varphi^{n-m-1}(C)) \cap \varphi^{n-m}(C)=\{0\}.\]
Suppose that one has $c\in C$
with
\[0\neq \varphi^{n-m}(c)\in C+\varphi(C)+ \cdots +\varphi^{n-m-1}(C).\]
Then one has a sequence $c_0, c_1, \ldots, c_{n-m-1} \in C$ with
\[\varphi^{n-m}(c)=\sum_{j=0}^{n-m-1}\varphi^j(c_j).\]
Let $j_0$ be the smallest integer with $c_{j_0}\neq 0$.
Then, since $\varphi$ is injective, one has
\[c_{j_0}= \varphi^{n-m-j_0}(c)-\sum_{k=j_0+1}^{n-m-1} \varphi^{k-j_0}(c_k)  \in \varphi(C).\]
This contradicts to the relation (\ref{equation:int}).
Hence $\Phi_{m, n}$ is injective, whence Claim \ref{Cl:D} is shown.

We now prove Claim \ref{Cl:2}.
Since an extension of type I \Cs-algebras is again of type I,
by the exact sequences in  (\ref{equation:extD}),
the \Cs-algebras $D_{m, n}$ must be of type I.
Now recall that separable type I \Cs-algebras are nuclear and satisfy the universal coefficient theorem \cite{RS}.
Since $D$ is the closure of the increasing union $\bigcup_{n\in \IN} D_{n, n}$, Claim \ref{Cl:2} holds true.

We now define the desired $\Gamma$-\Cs-algebra $A$.
Let $\sigma$ be the automorphism on $D$ satisfying
\begin{equation}\label{eq:sigma}
\sigma \circ \mu_{n}=\mu_{n}\circ \varphi~ (=\mu_{n-1})\quad {\rm for}\quad n\in \IZ.
\end{equation}
Since $\varphi$ is $\Gamma$-equivariant,
the action $\beta\colon \Gamma \acts C$ (see (\ref{eq:beta})) gives rise to the action $\gamma \colon \Gamma \acts D$
by the relations
\[\gamma_g\circ \mu_{n}=\mu_{n}\circ \beta_g \quad {\rm for}\quad n\in \IZ,~ g\in \Gamma.\]
(Indeed the same formulas give the actions $\IZ \acts B$ and $\Gamma \acts B$ by the universality of the inductive limit (\ref{eq:B}),
and the restrictions define the desired actions $\sigma$ and $\gamma$.)
Put
\[A:=D \rca{\sigma} \IZ.\]
Then $A$ is nuclear (see \cite{BO}, Theorem 4.2.4) and satisfies the universal coefficient theorem (see \cite{RS}).
Let $u\in \cM(A)$ denote the canonical implementing unitary element of $\sigma$.
Clearly $\gamma$ commutes with $\sigma$.
Hence $\gamma$ induces the action
$\alpha \colon \Gamma \acts A$ with
\[\alpha_g(au^n)=\gamma_g(a)u^n \quad {\rm for}\quad g \in \Gamma,~ a\in D,~ n\in \IZ.\]

We will show that $\alpha$ satisfies the desired conditions (\ref{item:a}) to (\ref{item:c}).

\begin{Cl}\label{Cl:ame}
The action $\alpha \colon \Gamma \acts A$
is amenable.
\end{Cl}

Since the action $\beta$ is amenable,
by the extension relations (\ref{equation:extD}) among $D_{m, n}$; $m\leq n$, which are clearly $\Gamma$-equivariant,
the restriction actions $\Gamma \acts D_{m, n}$; $m\leq n$, are amenable by \cite{OS}, Proposition 3.7.
Hence so is the inductive limit action $\gamma$.
By applying Theorem 5.1 of \cite{OS} to $\alpha$ and the gauge action 
\[\theta \colon\IT\acts D \rca{\sigma} \IZ =A; \quad \theta_z(a u^n)=z^n a u^n \quad {\rm~for~}z\in \IT,~a\in D,~ n\in \IZ\]  (which commutes with $\alpha$),
we conclude the amenability of $\alpha \colon \Gamma \acts A$.

We next show the following claim.

\begin{Cl}\label{Cl:type1}
The projection $\mu_{0}(\delta_e\otimes 1_{C(Z)})$ is properly infinite in $A$.
 \end{Cl}
Note that for any $s\in S$ and $j<0$, one has $\psi_{s, j}(1_{C(Z)})=1_{C(Z)}$.
Hence one has
\[\mu_{0}(\delta_e\otimes 1_{C(Z)})\oplus \mu_{0}(\delta_e\otimes 1_{C(Z)}) \precsim \sigma(\mu_{0}(\delta_e\otimes 1_{C(Z)}))\]
in $D$
by the definitions of $\varphi$ and $\sigma$ (see (\ref{eq:psisj}), (\ref{eq:psi}), (\ref{eq:phi}), and (\ref{eq:sigma})).
Moreover, one has
\[\sigma(\mu_{0}(\delta_e\otimes 1_{C(Z)})) =u \mu_{0}(\delta_e\otimes 1_{C(Z)}) u^\ast \sim \mu_{0}(\delta_e\otimes 1_{C(Z)})\]
in $A$. These two relations prove Claim \ref{Cl:type1}.
 
The next claim is the hardest part of the proof of Theorem \ref{Thm:1}.
 \begin{Cl}\label{Cl:type2}
The projection $\mu_{0}(\delta_e \otimes p_0)$ is not properly infinite in $A\rca{\alpha} \Gamma$.
 \end{Cl}
For convenience of the proof of Claim \ref{Cl:type2}, we introduce a few notations.
For a set $X$, we put $X^0:=\{\emptyset\}$.
For a finite sequence $\mathfrak{x}=(x_1, x_2, \ldots, x_n)\in X^n$; $n \geq 1$,
we define $\mathfrak{x}':=(x_2, \ldots, x_{n})\in X^{n-1}$, regarded as $\emptyset$ when $n=1$.
For $\fs, \ft, \fk, \fl \in X^n$, $n\geq 1$, and $1\leq j \leq n$, we denote their $j$th term by $s_j, t_j, k_j, l_j$ respectively.

We introduce the index set
\[\Omega:=\bigsqcup_{n\in \IN}(S^n\times \IZ^n ).\]
For $\omega\in \Omega_n:= S^n\times \IZ^n $; $n\in \IN$, we define $|\omega|:=n$
and call it the length of $\omega$.
We define an indexed family $(\cI(\omega))_{\omega\in \Omega}$ of finite subsets in $\IN$ by induction on the length as follows.
On $\Omega_0$,
we set
\[\cI(\emptyset, \emptyset):=\{0\}\subset \IN.\]
By induction on $n$, on $\Omega_n$, $n\geq 1$,
we define
\begin{equation}\label{eq:rec1}
\cI(\fs, \fk):=\kappa_{s_1, k_1}(\cI(\fs', \fk')) ~(=\nu(s_1, k_1, \cI(\fs', \fk')))
\end{equation}
for $(\fs, \fk)\in \Omega_n$.
By (\ref{eq:MvN}), (\ref{eq:psi}), and (\ref{eq:phi}), for $n\in \IN$ and $g\in \Gamma$, we have
\[\varphi^n(\delta_g \otimes p_0) \sim \psi^n(\delta_g \otimes p_0) \sim \bigoplus_{(\fs, \fk)\in \Omega_n} \delta_{s_1 \cdots s_n g}\otimes p_{\cI(\fs, \fk)}\]
in $\cM(C)$.

We observe that the indexed family $(\cI(\omega))_{\omega\in \Omega}$ satisfies the hypothesis of Proposition \ref{Prop:Ror}:
There is an injective map $N\colon \Omega \rightarrow \IN$ with $N(\omega) \in \cI(\omega)$ for $\omega \in \Omega$.
We define the desired map $N$ by induction on the length of $\omega\in \Omega$.
We first set $N(\emptyset, \emptyset):=0\in \cI(\emptyset, \emptyset)$.
Then, for $(\fs, \fk)\in \Omega_n$, $n\geq 1$, we inductively define
\begin{equation}\label{eq:rec2}
N(\fs, \fk):=\nu(s_1, k_1, N(\fs', \fk'))\in \kappa_{s_1, k_1}(\cI(\fs', \fk'))=\cI(\fs, \fk).
\end{equation}
We show the injectivity of $N$.
Observe that $N(\omega)\in \Lambda_{|\omega |}$ for $\omega\in \Omega$ by the choice of $\nu$ and the definition of $N$.
Hence, as $\Lambda_n$; $n\in \IN$, are pairwise disjoint,
it suffices to show the injectivity of $N$ on each $\Omega_n$.
We prove this by induction on $n$.
The case $n=0$ is trivial.
Suppose that the claim holds for $n-1 \in \IN$.
Pick two elements $(\fs, \fk), (\ft, \fl)\in \Omega_n$
with $N(\fs, \fk)=N(\ft, \fl)$.
Then, by the definition (\ref{eq:rec2}) of $N$ and the injectivity of $\nu$,
the equality implies
\[s_1 =t_1, \quad k_1=l_1, \quad N(\fs', \fk')=N(\ft', \fl').\]
By induction hypothesis, the last condition implies
$(\fs', \fk')=(\ft', \fl')$.
Thus we conclude $(\fs, \fk)=(\ft, \fl)$, whence $N$ is injective on $\Omega_n$.

Note that for two projections $p, q$ in $\cM(C)$,
one has $p\precsim q$ in $\cM(C)$ if and only if for all $g\in \Gamma$, $p_g \precsim q_g$ holds in $\cM(C(Z)\otimes \IK)$.
As a result, we obtain
\begin{equation}\label{eq:vn5}
\delta_e \otimes 1_{C(Z)}\not\precsim \bigoplus_{(\fs, \fk)\in \Omega, g\in \Gamma}\left( \delta_{s_1\cdots s_n g}\otimes p_{\cI(\fs, \fk)}\right) \sim \bigoplus_{(g, n)\in \Gamma \times \IN} \beta_g(\varphi^n(\delta_e \otimes p_0))
\end{equation}
in $\cM(C)$,
because the value of the middle projection at the unit $e \in \Gamma$ is Murray--von Neumann equivalent to
\[r:=\bigoplus_{(\fs, \fk)\in \Omega} p_{\cI(\fs, \fk)},\]
which satisfies $1_{C(Z)}\not\precsim r$ by Proposition \ref{Prop:Ror}.

We next show the relation
\begin{equation}\label{equation:vn1}\mu_{0}(\delta_e \otimes 1_{C(Z)}) \not\precsim \bigoplus_{(g, n)\in \Gamma \times \IZ} \gamma_g(\sigma^n(\mu_{0}(\delta_e \otimes p_0)))\end{equation}
in $\cM(B\otimes \IK)$ (which implies the same relation in $\cM(D\otimes \IK)$).
To lead to a contradiction, assume that (\ref{equation:vn1}) fails.
Then by Remark \ref{Rem:finite}, one has a finite subset $F \subset \Gamma\times \IZ$ with
\begin{equation}\label{eq:Bproj}
\mu_{0}(\delta_e \otimes 1_{C(Z)}) \precsim \bigoplus_{(g, n)\in F} \gamma_g(\sigma^n(\mu_{0}(\delta_e \otimes p_0)))
\end{equation}
in $D\otimes \IK$.
Because $B$ is the inductive limit
\[\varinjlim_{n\in \IZ}(\cM(C), \varphi),\]
by (\ref{eq:Bproj}),
for a sufficiently large $m\in \IN$ (with $n+m \geq 0$ for all $(g, n)\in F$),
one has
\begin{equation}\label{eq:vn3}\varphi^m (\delta_e \otimes 1_{C(Z)})  \precsim \bigoplus_{(g, n)\in F} \beta_g(\varphi^{n+m}(\delta_e \otimes p_0))
\end{equation}
in $\cM(C)$.
Since $e\in S$, for $k\in \IN$, we also have
\begin{equation}\label{eq:vn4}
\delta_e\otimes 1_{C(Z)}\precsim \varphi^k(\delta_e \otimes 1_{C(Z)}).
\end{equation}
The relations (\ref{eq:vn3}) and (\ref{eq:vn4}) imply
\[\delta_e\otimes 1_{C(Z)}\precsim  \bigoplus_{(g, n)\in F} \beta_g(\varphi^{n+m}(\delta_e \otimes p_0)),\]
which contradicts to the relation (\ref{eq:vn5}).

At the same time, by (\ref{eq:trivial}), we also have 
\begin{equation}\label{equation:vn2}
\mu_{0}(\delta_e \otimes 1_{C(Z)}) \precsim \mu_{0}((\delta_e \otimes p_0)\oplus (\delta_e \otimes p_0))
\end{equation}
in $D$.
Thus, by Lemma \ref{Lem:fp}, applied to $D\rca{\sigma\times \gamma}(\IZ \times \Gamma)=A\rca{\alpha}\Gamma$ with $p=\mu_0(\delta_e \otimes 1_{C(Z)})$, $q=\mu_0(\delta_e \otimes p_0) \in D$, and the group $\IZ \times \Gamma$,
the relations (\ref{eq:vn4}) and (\ref{equation:vn2}) prove Claim \ref{Cl:type2}.

We next show the following claim.
\begin{Cl}\label{Cl:simpleA}
The \Cs-algebra $A$ is simple.
\end{Cl}

We first show that for $k \geq 1$, $D$ does not contain a $\sigma^k$-invariant proper ideal.
Let $J$ be a nonzero ideal of $D$ with $\sigma^k(J)=J$.
Observe that $D$ is the closure of the union of the increasing \Cs-subalgebras
\[\mu_{n}(\cM(C)) \cap D,\quad n\in \IZ.\]
Hence $J$ contains a nonzero element in $\mu_{N}(\cM(C))$ for a sufficiently large $N\in \IZ$.
Then, as the ideal
\[\mu_{N}(C) \lhd \mu_{N}(\cM(C)) \cap D\]
is essential, 
we obtain $J \cap \mu_{N}(C) \neq \{0\}$.
Thus we have non-empty open subsets $T\subset \Gamma$, $U \subset Z$ with
$\mu_{N}(C_0(T\times U)\otimes \IK) \subset J$.
By the definition (\ref{eq:phi}) of $\varphi$ (see also (\ref{eq:psisj}), (\ref{eq:psi})), the inequality (\ref{eq:full}), and the assumption $e\in S$,
one has
\begin{equation}\label{eq:ideal1}
\ide{\varphi^k(C_0(T\times U)\otimes \IK)}{C} \supset C_0(T\times Z)\otimes \IK.
\end{equation}
Also, for any subset $M\subset \Gamma$, one has
\begin{equation}\label{eq:ideal2}
\ide{\varphi(C_0(M \times Z)\otimes \IK)}{C} \supset C_0((S\cdot M) \times Z)\otimes \IK.
\end{equation}
For $n \geq 1$, set
\[S^{\bullet n}:=\{s_1 s_2 \cdots s_n: s_k\in S {\rm~for~}k=1, 2, \ldots, n\}.\]
Then, for $l \geq 2$, it follows from (\ref{eq:ideal1}) and (\ref{eq:ideal2}) that
\begin{align*}
\ide{\varphi^{kl}(C_0(T\times U)\otimes \IK)}{C} &\supset \ide{\varphi^{k(l-1)}(C_0(T\times Z)\otimes \IK)}{C}\\
&\supset C_0((S^{\bullet k(l-1)} \cdot T) \times Z) \otimes \IK.
\end{align*}
Since $S \subset S^{\bullet k}$, the subset $S^{\bullet k}$ also generates $\Gamma$ as a semigroup.
Since $J$ is $\sigma^k$-invariant, we obtain
\[\mu_{N}(\varphi^{kl}(C_0(T \times U)\otimes \IK)) \subset \sigma^{kl}(J) = J\]
for all $l \in \IN$.
Hence 
\begin{equation}\label{eq:J}
\mu_{N}(C)\subset J.
\end{equation}
Since $J$ is $\sigma^k$-invariant, the relation (\ref{eq:J}) yields
\begin{equation}\label{eq:J2}
\mu_{kn+N}(C)\subset J \quad {\rm~for~} n\in \IZ.
\end{equation}
For all $j\in \IN$, one has $\ide{\varphi^j(C)}{C}=C$ by (\ref{eq:ideal2}).
Hence, for any $l, m\in \IN$ with $l\geq m$, one has
\[\mu_{l}(C)\subset \ide{\mu_{m}(C)}{D}.\]
Thus the relation (\ref{eq:J2}) implies $J=D$.

It follows from the relations (\ref{equation:vn1}) and (\ref{equation:vn2}) that
\[\sigma^n (\mu_{0}(\delta_e \otimes p_0)) \not \sim \mu_{0}(\delta_e \otimes p_0)\]
in $D$ for $n\in \IZ\setminus \{0\}$.
By a characterization of proper outerness in \cite{OP}; cf.~ the argument in Page 132 of \cite{Ror},
we conclude the proper outerness of $\sigma^k$ for $k \geq 1$.
Thus Claim \ref{Cl:simpleA} follows from Kishimoto's theorem \cite{Kis}.

We next show the following claim.
\begin{Cl}\label{Cl:outer}
The action $\alpha \colon \Gamma \acts A$ is pointwise outer.
\end{Cl}
Note that the simplicity of $A\rca{\alpha}\Gamma$
follows from Claims \ref{Cl:simpleA}, \ref{Cl:outer}, and Kishimoto's theorem \cite{Kis}.
Hence this will also imply the finiteness of $\mu_{0}(\delta_e \otimes p_0)$ in $A\rca{\alpha} \Gamma$
by Claim \ref{Cl:type2} and \cite{Cun} (see also Proposition 6.11.3 in \cite{BlaK}).

To show Claim \ref{Cl:outer},
we consider the faithful left regular representation
\[\Phi\colon A=D\rca{\sigma} \IZ \rightarrow \cM(D\otimes \IK(\ell^2(\IZ)))\]
as in the proof of Lemma \ref{Lem:fp}.
Let $q\in \IK(\ell^2(\IZ))$ be a minimal projection.
Then the relation (\ref{equation:vn1}) implies
\begin{align}\label{align1}
\mu_{0}(\delta_e\otimes 1_{C(Z)}) \otimes q &\not\precsim \bigoplus_{g\in \Gamma, n\in \IZ} \alpha_g(\sigma^n(\mu_{0}(\delta_e\otimes p_0)))\otimes q	\\
&\sim \bigoplus_{g\in \Gamma} \Phi(\alpha_g(\mu_{0}(\delta_e \otimes p_0)))\notag \end{align}
in $\cM(D\otimes \IK(\ell^2(\IZ)))$.
Also, the relation (\ref{equation:vn2}) implies
\begin{align}\label{align2}\mu_{0}(\delta_e\otimes 1_{C(Z)})\otimes q &\precsim \Phi(\mu_{0}(\delta_e \otimes 1_{C(Z)}))\\
&\precsim \Phi(\mu_{0}(\delta_e \otimes p_0)\oplus \mu_{0}(\delta_e \otimes p_0))\notag 
\end{align}
in $\cM(D\otimes \IK(\ell^2(\IZ)))$.
Thus, for $g\in \Gamma \setminus \{e\}$, the two relations (\ref{align1}) and (\ref{align2}) force
 \[\mu_{0}(\delta_e \otimes p_0)\not\sim \alpha_g(\mu_{0}(\delta_e \otimes p_0))\]
 in $A\otimes \IK$. This proves Claim \ref{Cl:outer}.

We next show the following claim. This will complete the proof of the first part of Theorem \ref{Thm:1}.
\begin{Cl}\label{Cl:uct}
The \Cs-algebra $A\rca{\alpha}\Gamma$ is nuclear and satisfies the universal coefficient theorem \cite{RS}.
\end{Cl}
To show Claim \ref{Cl:uct},
observe that
$C \rca{\beta} \Gamma$
is of type I, as it is isomorphic to $\IK(\ell^2(\Gamma))\otimes C(Z) \otimes \IK$.
The $\Gamma$-equivariant split short exact squences among $D_{m, n}$; $m\leq n$ in (\ref{equation:extD})
give rise to the short exact sequences
\[\{0\}\rightarrow D_{m+1, n}\rca{\gamma}\Gamma \rightarrow D_{m, n}\rca{\gamma}\Gamma \rightarrow C\rca{\beta}\Gamma \rightarrow \{0\}.\]
Thus each $D_{m, n}\rca{\gamma} \Gamma$ is of type I.
Hence $D\rca{\gamma}\Gamma$ is the closure of an increasing union of type I \Cs-algebras.
Since $A\rca{\alpha} \Gamma\cong (D\rca{\gamma} \Gamma)\rca{\sigma}\IZ$,
Claim \ref{Cl:uct} follows by \cite{RS}.

Lastly, we prove the last claim of Theorem \ref{Thm:1}; unital examples for the free groups.
Since $A$ is simple and contains a properly infinite projection,
it follows from \cite{Cun} that
$A$ contains properly infinite (nonzero) projections $p, q$ with $[p]_0=0$ in K$_0(A)$ and $pq=0$.
Then, by \cite{Cun}, the projections $(\alpha_g(p))_{g\in \Gamma}$ are mutually unitary equivalent in $A$.

Now let $\Gamma$ be a countable free group, and let $\mathcal{F}$ be its free basis.
For each $s\in \mathcal{F}$, we pick a unitary element $u_s\in \cM(A)$ with $u_s \alpha_s(p) u_s^\ast =p$.
Now define $\hat{\alpha} \colon \Gamma \acts pAp$ to be
\[\hat{\alpha}_s(x):=u_s\alpha_s(x)u_s^\ast\]
for $s\in \mathcal{F}$ and $x\in pAp$. (By the freeness, indeed this extends to an action $\hat{\alpha}$ of $\Gamma$.)
One has an isomorphism
\[\Theta \colon pAp \rca{\hat{\alpha}} \Gamma \rightarrow p(A \rca{\alpha} \Gamma)p\]
given by $\Theta(xs)=xu_s s$ for $x\in pAp$ and $s\in \mathcal{F}$.
Hence $\hat{\alpha}$ satisfies the condition (\ref{item:c}) in Theorem \ref{Thm:1}.
The other two conditions (\ref{item:a}) and (\ref{item:b}) follow from the corresponding properties of $\alpha$ and $A$.
\end{proof}
\begin{Rem}\label{Rem:simple}
When an action $\alpha \colon \Gamma \acts A$ is amenable,
the normal subgroup $\{g\in \Gamma: \alpha_g \text{ is inner}\} \lhd \Gamma$
is amenable.
In particular, when $\Gamma$ has trivial amenable radical,
amenable actions must be pointwise outer.

Thus, for any non-trivial group $\Gamma$,
by embedding $\Gamma$ into a group $\Upsilon$ with trivial amenable radical (e.g., $\Gamma \ast \IZ$),
applying the construction in the proof of Theorem \ref{Thm:1} to $\Upsilon$ instead of $\Gamma$,
and then restricting the action to $\Gamma$,
we also obtain the desired $\Gamma$-action
whose pointwise outerness is clear by the reason in the previous paragraph.
\end{Rem}
\begin{Rem}
Although we do not use the following observation in this article, 
it would be useful to distinguish some \Cs-dynamical systems.

The resulting \Cs-algebra $A \rca{\alpha} \Gamma$ in the proof of Theorem \ref{Thm:1} further satisfies the following property:
For any non-empty finite subsets $F\subset \Gamma$ and $I\subset \Lambda_0$ with $|I| \geq |F|$,
the projection 

\[\mu_{0}(\chi_F \otimes p_I) =\sum_{g\in F} \alpha_g(\mu_{0}(\delta_e \otimes p_I) )~ (\sim \bigoplus_F \mu_0(\delta_e \otimes p_I) {\rm~ in~} A\rca{\alpha} \Gamma)\]
is finite in $A \rca{\alpha} \Gamma$.

To see this, choose an injective map $n_\bullet \colon F \rightarrow I$.
We define indexed families of finite subsets $(\mathcal{J}(\omega))_{\omega\in \Omega}$ of $\IN$ and $(M(\omega, s))_{\omega \in \Omega} \subset \IN$ for each $s\in F$
by the same recurrence relations as in the definitions of $\cI(\omega)$ and $N(\omega)$
(see (\ref{eq:rec1}), (\ref{eq:rec2})), 
but with the different initial values
\[\mathcal{J}(\emptyset, \emptyset)=I,\quad M(\emptyset, \emptyset, s)=n_s.\]
Then it follows that $M\colon \Omega \times F \rightarrow \IN$ is injective and
satisfies $M(\omega, s) \in \mathcal{J}(\omega)$ for $(\omega, s)\in \Omega \times S$.
Now observe that
\begin{equation}\label{eq:Rem}
\bigoplus_{(g, n)\in \Gamma \times \IN} \beta_g(\varphi^n(\chi_F \otimes p_{I}))
\sim \bigoplus_{\substack{(\fs, \fk)\in \Omega,\\ (t, g)\in F \times \Gamma}} \delta_{s_1\cdots s_n t g}\otimes p_{\mathcal{J}(\fs, \fk)}
\end{equation}
in $\cM(C)$.
Obviously for each $\fs\in S^n$ and $s\in \Gamma$, there are exactly $|F|$ elements $(t, g)$ in $F \times \Gamma$
with $s_1 s_2 \cdots s_n tg =s$.
Thus the value of the right hand side in (\ref{eq:Rem}) at any $s \in F$ is Murray--von Neumann equivalent to
\[r:=\bigoplus_{(\omega, s) \in \Omega \times F}p_{\mathcal{J}(\omega)}.\]
By Proposition \ref{Prop:Ror}, one has $1_{C(Z)} \not \precsim  r$ in $\cM(C(Z)\otimes \IK)$.
Thus
\[\chi_F\otimes 1_{C(Z)} \not \precsim \bigoplus_{(g, n)\in \Gamma \times \IN} \beta_g(\varphi^n(\chi_F \otimes p_{I}))\]
in $\cM(C)$.
Now by the same proof as Claim \ref{Cl:type2} in Theorem \ref{Thm:1},
we conclude the finiteness of $\mu_{0}(\chi_F \otimes p_I)$ in $A \rca{\alpha} \Gamma$.
\end{Rem}

 \section{Theorem \ref{Thm:2}}
 For the proof of Theorem \ref{Thm:2}, we record the following basic fact.
 This is the non-commutative version of the (well-known) last statement of Theorem 5.1.7 in \cite{BO}.
\begin{Lem}\label{Lem:sep}
Let $\Gamma$ be a countable group.
Let $\alpha \colon \Gamma \acts A$ be an amenable action on a unital \Cs-algebra $A$.
Let $\mathcal{S}\subset A$ be a separable subset.
Then there is a unital separable $\Gamma$-\Cs-subalgebra
$B\subset A$ such that $\mathcal{\mathcal{S}}\subset B$ and
that the action $\Gamma \acts B$ is amenable.
\end{Lem}
\begin{proof}
Take a (norm) compact subset $K_0 \subset A$ whose closed linear span contains $\mathcal{S}$.
Fix an increasing sequence $(F_n)_{n\in \IN}$ of finite subsets of $\Gamma$ whose union is $\Gamma$.
Fix a sequence $(\epsilon_n)_{n\in \IN}$ in $(0, 1)$ converging to $0$.
By the amenability of $\alpha$ (see Definition 2.11 and Theorem 3.2 in \cite{OS}), one can find $\xi_0 \in \ell^2(\Gamma, A)$ with
\[\| \ipa{\xi_0, \tilde{\alpha}_{g}(\xi_0)} -1 \|\leq \epsilon_0, \quad \|a\xi_0 - \xi_0 a\| \leq \epsilon_0\]
for $g\in F_0$ and $a\in K_0$.
Define $B_1$
to be the $\Gamma$-\Cs-subalgebra of $A$ generated by $K_0$ and the image of $\xi_0$. Fix a compact subset $K_1 \subset B_1$
that spans a dense subspace of $B_1$ with $K_0 \subset K_1$.
Again by the amenability of $\alpha$,
one has 
 $\xi_1 \in \ell^2(\Gamma, A)$ with
\[\|\ipa{\xi_1, \tilde{\alpha}_{g}(\xi_1)} -1 \| \leq \epsilon_1, \quad \|a\xi_1 - \xi_1 a\| \leq \epsilon_1\]
for $g\in F_1$ and $a\in K_1$.
 
 We iterate this argument by induction.
 Then we get an increasing sequence $(B_n)_{n=1}^\infty$ of separable $\Gamma$-\Cs-subalgebras;
 an increasing sequence $(K_n)_{n\in \IN}$ of compact subsets of $A$;
 and a sequence $(\xi_n)_{n\in \IN}$ in $\ell^2(\Gamma, A)$ satisfying the following conditions for $n\in \IN$:
 \begin{itemize}
 \item The closed linear span of $K_{n+1}$ is equal to $B_{n+1}$.
 \item $B_{n+1}$ contains the image of $\xi_n$.
 \item For $g\in F_n$ and $a\in K_n$, one has
 \[\|\ipa{\xi_n, \tilde{\alpha}_{g}(\xi_n)} -1 \|\leq {\epsilon_n}, \quad \|a\xi_n - \xi_n a\| \leq \epsilon_n.\]

 \end{itemize}
The closure $B$ of the increasing union $\bigcup_{n=1}^\infty B_n$ is a unital separable $\Gamma$-\Cs-subalgebra of $A$ containing $\mathcal{S}$.
The sequence $(\xi_n)_{n\in \IN}$, which sits in $\ell^2(\Gamma, B)$ by the second condition,
witnesses the amenability of $\Gamma \acts B$ by the other two conditions.
\end{proof}

\begin{proof}[Proof of Theorem \ref{Thm:2}]
We apply the construction in the proof of Theorem \ref{Thm:1} with $S=\Gamma$.
(In fact the condition $S^{\bullet n}=\Gamma$ for some $n \geq 1$ is enough.)
Then the resulting inductive limit \Cs-algebra $B$ (\ref{eq:B}) therein is also simple.
This follows from the following property of $\varphi$.
\begin{Cl}
When $S=\Gamma$, for any nonzero $f\in C$,
one has $\ide{\varphi(f)}{C}=C$.
Moreover, $\varphi^2(f)$ is full in $\cM(C)$.
\end{Cl}
Indeed, for any nonzero $f\in C$,
by the definition of $\varphi$ (see (\ref{eq:psisj}), (\ref{eq:psi}), (\ref{eq:phi})) and the inequality (\ref{eq:full}),
one has 
\[\min_{g\in \Gamma, z\in Z} \|\varphi_g(f)(z)\| \geq \frac{\|f\|}{2} >0.\]
This proves the first statement. 
Then, since $1_C \precsim \varphi(\delta_e \otimes 1_{C(Z)})$ in $\cM(C)$ (by the choice $S=\Gamma$), the second statement also holds true.

Since $\Gamma$ is exact, the right translation action
$\Gamma \acts \cM(C)$
is amenable.
(Indeed one has
a unital $\Gamma$-equivariant embedding of $\ell^\infty(\Gamma)$ into the center of $\cM(C) =\ell^\infty(\Gamma, \cM(C(Z)\otimes \IK))$.
By the exactness of $\Gamma$,
the right translation action $\Gamma \acts \ell^\infty(\Gamma)$ is amenable by Ozawa's theorem \cite{Oz}.
Hence $\beta$ is indeed (strongly) amenable.)
Hence the inductive limit action
\[\Gamma \acts B=\varinjlim_{n\in \IZ}(\cM(C), \varphi)\]
is amenable as well.
Note that $B$ is non-exact as $\cM(\IK) \subset B$.

By applying Lemma \ref{Lem:sep} to a separable, non-exact \Cs-subalgebra $\mathcal{S}$ of $B$ containing $\mu_{0}(C)$
and isometry elements $v, w$
with $w^\ast v=0$,
we obtain a unital separable non-exact $\Gamma$-\Cs-subalgebra $B_0 \subset B$
whose restriction $\Gamma$-action is amenable
and contains $\mathcal{S}$.
By applying Lemma 5.7 of \cite{Ror} (\cite{Bla}) to $B_0$,
we obtain a simple separable \Cs-subalgebra $C_0\subset B$ containing $B_0$.
By iterating these constructions alternatively, we obtain an increasing sequence 
\[B_0 \subset C_0 \subset B_1 \subset C_1 \subset \cdots \subset B_n \subset C_n\subset \cdots\]
of separable \Cs-subalgebras in $B$
with the conditions
\begin{itemize}
\item each $B_n$ is a $\Gamma$-\Cs-subalgebra of $B$ whose restriction $\Gamma$-action is amenable,
\item each $C_n$ is simple.
\end{itemize}

Consider the closure $A$ of $\bigcup_{n\in \IN} B_n=\bigcup_{n\in \IN} C_n$.
(This $A$ is different from the \Cs-algebra in the proof of Theorem \ref{Thm:1}.)
Then clearly $A$ is a unital, simple, separable, properly infinite, non-exact $\Gamma$-\Cs-subalgebra of $B$.
Moreover the restriction action, denoted by $\alpha \colon \Gamma \acts A$, is amenable (see \cite{OS}, Proposition 3.7).
By the formula (\ref{equation:vn1}) in the proof of Theorem \ref{Thm:1}, we have
\[\mu_{0}(\delta_e\otimes 1_{C(Z)})\not\precsim \bigoplus_{g\in \Gamma} \alpha_g(\mu_{0}(\delta_e\otimes p_0))\]
in $\cM(B\otimes \IK)$ and therefore in $\cM(A\otimes \IK)$.
We also have
\[\mu_{0}(\delta_e\otimes 1_{C(Z)})\precsim \mu_{0}(\delta_e\otimes p_0) \oplus \mu_{0}(\delta_e\otimes p_0)\]
in $A\otimes \IK$ because it holds in $\mu_{0}(C) \subset A$ by (\ref{eq:trivial}).
These two relations in particular show that $\alpha$ is pointwise outer.
Hence $A\rca{\alpha} \Gamma$ is simple by Kishimoto's theorem \cite{Kis}.
The above relations also show the finiteness of $\mu_{0}(\delta_e\otimes p_0)$ in $A\rca{\alpha} \Gamma$ by Lemma \ref{Lem:fp} and \cite{Cun}.
\end{proof}
We close the article by giving several remarks on the construction in the proof of Theorem \ref{Thm:2}.
We use the notations in the proof of Theorem \ref{Thm:2}.
\begin{Rem}
Since K$_0(\cM(C))={\rm K}_1(\cM(C))=\{0\}$ (see \cite{BlaK}, Propostion 12.2.1),
one has K$_0(B)={\rm K}_1(B)=\{0\}$.
Hence by slightly modifying the construction in the proof of Theorem 2, one can arrange the resulting \Cs-algebra $A$ therein
to further satisfy ${\rm K}_0(A)={\rm K}_1(A)=\{0\}$.

Indeed, for any unital separable \Cs-subalgebra $C\subset B$, it is not hard to find
a separable \Cs-subalgebra $C\subset D \subset B$ with K$_0(D)={\rm K}_1(D)=\{0\}$.
We additionally apply this construction to each $C_n$ in the proof of Theorem \ref{Thm:2},
and then require $B_{n+1}$ to contain $D_n$.
As a result, we obtain a new increasing sequence
\[B_0 \subset C_0 \subset D_0 \subset  B_1 \subset C_1 \subset D_1 \subset  \cdots \subset B_n \subset C_n\subset D_n \subset \cdots\]
of separable \Cs-subalgebras
with
\begin{itemize}
\item $B_n$ and $C_n$ have the same properties as in the proof,
\item K$_0(D_n)={\rm K}_1(D_n)=\{0\}$ for $n\in \IN$.
\end{itemize}
By the continuity of the $K$-theory,
the closure of $\bigcup_{n\in \IN} B_n= \bigcup_{n\in \IN} C_n =\bigcup_{n\in \IN} D_n$
gives the desired $\Gamma$-\Cs-algebra.
\end{Rem}
\begin{Rem}\label{Rem:central}
For a given action $\Gamma \acts X$ on a compact metrizable space
with a dense orbit $\Gamma \cdot x$,
the orbit map $g \mapsto g\cdot x$ induces the unital $\Gamma$-equivariant embedding $\iota \colon C(X) \rightarrow \ell^\infty(\Gamma)$.
Using this property, one can further arrange $A$ to admit a unital $\Gamma$-equivariant embedding
of $C(X)$ into its central sequence algebra.
Indeed one can choose $A$ in Theorem \ref{Thm:2} to further satisfy $\mu_{n}(\iota(C(X)))\subset A$ for all $n\in \IN$.
The sequence of the maps $(\mu_{n}\circ \iota)_{n\in \IN}$ defines the desired map. 

In particular, when $\Gamma$ is finite or a free abelian group of finite rank,
applying this to a free metrizable profinite action, one can arrange the action $\alpha \colon \Gamma \acts A$
as in Theorem \ref{Thm:2} with the Rohlin property \cite{KisR}, \cite{IzuR}, \cite{Mat}.

Also, for any countable exact group $\Gamma$, by choosing $\Gamma \acts X$ to be free,
one can arrange $\alpha \colon \Gamma \acts A$ to be
centrally free in the sense of \cite{SuzCMP}, Definition 4.1 (that is, the condition therein holds with $B=A$).
\end{Rem}
\begin{Rem}
For a given action $\eta \colon \Gamma \acts D$ on a unital separable \Cs-algebra,
one can arrange $\alpha \colon \Gamma \acts A$ in Theorem \ref{Thm:2} 
to contain $D$ as a unital $\Gamma$-\Cs-algebra.
To see this, observe that, by choosing $B_0$ large enough,
one can arrange the resulting action $\alpha \colon \Gamma \acts A$
with the following property: the fixed point algebra $A^\Gamma$ contains
$D\rca{\eta} \Gamma$ as a unital \Cs-subalgebra.
(This is possible, since there are unital embeddings 
$D\rca{\eta} \Gamma \rightarrow \cM(\IK)$ and $\cM(\IK) \rightarrow B^\Gamma$.)
For $g\in \Gamma$, let $u_g \in D\rca{\eta} \Gamma \subset A^\Gamma$ be its implementing unitary element.
Then $u\colon \Gamma \rightarrow \mathcal{U}(A)$ is an $\alpha$-$1$-cocycle.
Thus one has a new action $\alpha^u \colon \Gamma \acts A$ given by
\[\alpha^u_g:=\ad(u_g)\circ \alpha_g \quad {\rm~for~}g\in \Gamma.\]
Clearly this gives the desired action.
\end{Rem}
\begin{Rem}While $B$ is non-separable,
the proof of Theorem \ref{Thm:2} also shows that $\Gamma \acts B$ is amenable and pointwise outer,
and that both $B$ and $B \rtimes_{\rm r}\Gamma$ have a finite and an infinite projection.
\end{Rem}

\subsection*{Acknowledgements}
This work was supported by JSPS KAKENHI (Grant-in-Aid for Early-Career Scientists)
Grant Numbers JP19K14550, JP22K13924.
The author is grateful to the referee for his/her careful reading and many suggestions which were very helpful to improve the exposition of the article.


\begin{thebibliography}{99}
\bibitem{Ana79}C.~ Anantharaman-Delaroche,
{\it Action moyennable d’un groupe localement compact sur une alg{\`e}bre de von Neumann.} Math.~ Scand.~ {\bf 45} (1979), 289--304.
\bibitem{AnaT}C.~ Anantharaman-Delaroche, {\it Amenability and exactness for dynamical systems and their \Cs-algebras.} Trans.~ Amer.~ Math.~ Soc.~ {\bf 354} (2002), no. 10, 4153--4178.
\bibitem{Bla}B.~ Blackadar, {\it Weak expectations and nuclear \Cs-algebras.} Indiana Univ.~ Math.~ J.~{\bf 27} (1978), 1021--1026.

\bibitem{BlaK}B.~ Blackadar, {\it K-Theory for Operator Algebras.} MSRI Monographs, volume 5,
Cambridge University Press, Cambridge, second edition, 1998.
\bibitem{BO} N.~ P.~ Brown, N.~Ozawa, {\it \Cs-algebras and finite-dimensional approximations.} Graduate Studies in Mathematics {\bf 88}. American Mathematical Society, Providence, RI, 2008.
\bibitem{Cun}J.~ Cuntz, {\it {\rm K}-theory for certain \Cs-algebras.} Ann.~ of Math.~ {\bf 113} (1981), 181--197.

\bibitem{CGS} J.~ Carri{\'o}n, J.~Gabe, C.~ Schafhauser, A.~ Tikuisis, S.~ White.
{\it Classifying $\ast$-homomorphisms I: unital simple nuclear \Cs-algebras.}
arXiv:2307.06480v3.

\bibitem{ELN}G.~ A.~ Elliott, C.~ G.~ Li, Z.~ Niu, {\it Remarks on Villadsen algebras.}
J.~ Funct.~ Anal.~ {\bf 287} (2024), Issue 7, 110547.
\bibitem{GS2} J.~ Gabe, G.~ Szab\'o, {\it The dynamical Kirchberg--Phillips theorem.}
Acta Math.~ {\bf 232} (2024) 1--77.
\bibitem{GG} E.~ Gardella, S.~ Geffen, J.~ Kranz, P.~ Naryshkin, A.~ Vaccaro, {\it Tracially amenable actions and purely infinite crossed products.}
Math.~ Ann.~ {\bf 390} (2024), no. 3, 3665--3690.
\bibitem{HP1}I.~Hirshberg, N.~C.~Phillips, {\it A simple nuclear \Cs-algebra with an internal
asymmetrical.} Anal.~ PDE, {\bf 16} (3) (2023)711--745.
\bibitem{HP2}I.~Hirshberg, N.~C.~Phillips, {\it Simple AH algebras with the same Elliott invariant and radius of comparison.} Preprint, arXiv:2312.11203.
\bibitem{IzuR}M.~ Izumi, {\it Finite group actions on \Cs-algebras with the Rohlin property I.} Duke Math.~ J.~ {\bf 122} (2004), no. 2, 233--280.

\bibitem{Kis}A.~Kishimoto, {\it Outer automorphisms and reduced crossed products of simple \Cs-algebras.} Comm.~ Math.~ Phys.~ {\bf 81} (1981), no. 3, 429--435.
\bibitem{KisR}A.~ Kishimoto, {\it The Rohlin property for automorphisms of UHF algebras.} J.~ Reine Angew.~ Math.~ {\bf 465} (1995), 183--196.
\bibitem{LR} X.~ Li, A.~I.~ Raad, {\it Constructing \Cs-diagonals in AH-algebras.} Trans.~ Amer.~ Math.~ Soc.~{\bf 376} (2023), no.12, 8857--8875.
\bibitem{Mat}H.~Matui, {\it $\mathbb{Z}^N$-actions on UHF algebras of infinite type.}
J.~ Reine Angew.~ Math.~ {\bf 657} (2011), 225--244. 
\bibitem{Niu}Z.~Niu, {\it Comparison radius and mean topological dimension: Rokhlin property, comparison of open sets, and subhomogeneous \Cs-algebras.}
Journal d'Analyse Math., {\bf 146} (2022), no. 2, 595--672.
\bibitem{OP}
D.~ Olesen, G.~ K.~ Pedersen, {\it Applications of the Connes Spectrum to \Cs-dynamical Systems
{\rm III}.} J.~ Funct.~ Anal.~ {\bf 45} (1981), no. 3, 357--390.
\bibitem{Oz}N.~ Ozawa, {\it Amenable actions and exactness for discrete groups.}
C.~ R.~ Acad.~ Sci.~ Paris Ser.~ I Math., {\bf 330} (8) (2000), 691--695.
\bibitem{OS}N.~ Ozawa, Y.~ Suzuki, {\it On characterizations of amenable \Cs-dynamical systems and new
examples.} Selecta Math.~ (N.S.), {\bf 27} (2021), Article number: 92 (29 pages).

 \bibitem{Rorbook}M.~R{\o}rdam, {\it Classification of nuclear \Cs-algebras.} Encyclopaedia of Mathematical Sciences
{\bf 126} (2002), 1--145.
 \bibitem{Ror}M.~R{\o}rdam, {\it A simple \Cs-algebra with a finite and an infinite projection.} Acta Math.~ {\bf 191} (2003), 109--142.
\bibitem{RS}J.~ Rosenberg, C.~ Schochet, {\it The K\"{u}nneth theorem and the universal coefficient theorem for Kasparov's generalized K-functor.} Duke Math.~ J.~ {\bf 55} (1987), no.~ 2, 431--474.
\bibitem{Suzeq}Y.~Suzuki, {\it Simple equivariant \Cs-algebras whose full and reduced crossed products coincide.}
J.~ Noncommut.~ Geom. {\bf 13} (2019), 1577--1585.
\bibitem{SuzIMRN}Y.~Suzuki, {\it Almost finiteness for general \'etale groupoids and its applications to stable rank of crossed products.}
Int.~ Math.~ Res.~ Not., {\bf 2020} (2020), 6007--6041.
\bibitem{SuzRig}Y.~Suzuki, {\it Rigid sides of approximately finite dimensional simple operator algebras in non-separable category.} Int.~ Math.~ Res.~ Not.~ {\bf 2021} (2021), 2166--2190.

\bibitem{SuzCMP} Y.~ Suzuki, {\it Complete descriptions of intermediate operator algebras by intermediate extensions of dynamical systems.} Commun.~ Math.~ Phys.~ {\bf 375} (2020), 1273--1297.
\bibitem{Suz21}Y.~Suzuki, {\it Equivariant $\mathcal{O}_2$-absorption theorem for exact groups.} Compos. Math. \textbf{157} (2021), 1492--1506.

\bibitem{SuzMAAN}Y.~Suzuki, {\it Non-amenable tight squeezes by Kirchberg algebras.}
Math.~ Ann.~ {\bf 382} (2022), 631--653.
\bibitem{Suzf}Y.~ Suzuki, {\it Every countable group admits amenable actions on stably finite simple \Cs-algebras.} To appear in Amer.~ J.~ Math., 2204.04480v2.
\bibitem{SuzP}Y.~Suzuki, {\it Amenable actions on finite simple \Cs-algebras arising from flows on Pimsner algebras.}
To appear in M{\"u}nster J. Math., special issue in honour of Eberhard Kirchberg (invited), arXiv:2305.13056.
\bibitem{Tom}A.~S.~ Toms, {\it On the classification problem for nuclear \Cs-algebras.} Ann.~ of Math.~ (2), {\bf 167} (2008), 1029--1044.
\bibitem{Vil1}J.~ Villadsen, {\it Simple \Cs-algebras with perforation.} J.~ Funct.~ Anal.~{\bf 154} (1998), no. 1, 110--116.
\bibitem{Vil2}J.~ Villadsen, {\it On the stable rank of simple \Cs-algebras.} J.~ Amer.~ Math.~ Soc.~ {\bf 12} (1999), no. 4,
1091--1102.
\bibitem{Win} W.~ Winter, {\it Structure of nuclear \Cs-algebras: From quasidiagonality to classification, and back again.} Proc.~ Int.~ Congr.~ Math.~ (2017), 1797--1820. 
\end{thebibliography}
\end{document}